\documentclass[3p]{elsarticle}

\usepackage{amsmath,amsthm,amssymb,mathrsfs}
\usepackage{enumerate}
\usepackage{txfonts,dsfont}
\usepackage[breaklinks,colorlinks]{hyperref}

\newdimen\AAdi%
\newbox\AAbo%
\def\AAk#1#2{\s_etbox\AAbo=\hbox{#2}\AAdi=\wd\AAbo\kern#1\AAdi{}}%
\def\AAr#1#2#3{\s_etbox\AAbo=\hbox{#2}\AAdi=\ht\AAbo\raise#1\AAdi\hbox{#3}}%
\font\tenmsb=msbm10 at 12pt
\font\sevenmsb=msbm7 at 8pt
\font\fivemsb=msbm5 at 6pt
\newfam\msbfam
\textfont\msbfam=\tenmsb
\scriptfont\msbfam=\sevenmsb
\scriptscriptfont\msbfam=\fivemsb
\def\Bbb#1{{\tenmsb\fam\msbfam#1}}

\newcommand{\beq}{\begin{equation}}
\newcommand{\eeq}{\end{equation}}
\newcommand{\beqr}{\begin{eqnarray}}
\newcommand{\eeqr}{\end{eqnarray}}
\newcommand{\ba}{\begin{array}}
\newcommand{\ea}{\end{array}}

\newtheorem{thm}{Theorem}[section]

\newtheorem{cor}[thm]{Corollary}
\newtheorem{rem}[thm]{Remark}

\newtheorem{eg}{Example}[section]
\newtheorem{conj}{Conjecture}

\numberwithin{equation}{section}

\def \vs{\vspace*{0.2cm}}

\def \ds{\displaystyle}
\def \p{\partial}

\def \l{\lambda}

\def \D{\Delta}
\def \d{\delta}

\def \s{\sigma}
\def \e{\varepsilon}
\def \a{\alpha}

\def\la{\langle}\def\ra{\rangle}

\def\cal{\mathcal}

\def\n{\nabla}

\def\M1{{{\cal M}}_1}

\def\n{\nabla}

\def\div{\hbox{div}\,}

\def\circledwedge{\setbox0=\hbox{$\bigcirc$}\relax \mathbin {\hbox
to0pt{\raise.5pt\hbox to\wd0{\hfil $\wedge$\hfil}\hss}\box0 }}

\journal{}
\allowdisplaybreaks

\begin{document}

\begin{frontmatter}

\title{ Rigidity theorems of spacelike entire self-shrinking graphs in the pseudo-Euclidean space \tnoteref{QS}}

\author[whu1,whu2]{Hongbing Qiu}
\ead{hbqiu@whu.edu.cn}

\author[whu1,whu2]{Linlin Sun\corref{sll1}}
\ead{sunll@whu.edu.cn}

\tnotetext[QS]{This research is partially supported by the National Natural Science Foundation of China (Nos. 11771339, 11801420), Fundamental Research Funds for the Central Universities (No. 2042019kf0198) and the Youth Talent Training Program of Wuhan University. The authors would like to thank Professor Y. L. Xin for his valuable suggestions and constant support. The first author also would like to express his gratitude to Professor Tobias H. Colding for his invitation, to MIT for their hospitality. The second author thanks the Max Planck Institute for Mathematics in the Sciences for good working conditions when this work carried out.}

\address[whu1]{School of Mathematics and Statistics, Wuhan University, Wuhan 430072, China}
\address[whu2]{Hubei Key Laboratory of Computational Science, Wuhan University, Wuhan, 430072, China
}

\cortext[sll1]{Corresponding author.}

\begin{abstract}
In this paper, we firstly establish a new volume growth estimate for spacelike entire graphs in the pseudo-Euclidean space $\mathbb{R}^{m+n}_n$. 
Then by using this volume growth estimate and the Co-Area formula, we prove various rigidity results for spacelike entire self-shrinking graphs.

\end{abstract}


\end{frontmatter}

\section{Introduction}
The pseudo-Euclidean space $\mathbb{R}^{m+n}_{n}$ of index $n$ is the linear
space $\mathbb{R}^{m+n}$ with coordinates\ $\left(x_1, x_2, \dotsc, x_{m+n}\right)$ and indefinite  metric
\begin{align*}
  ds^2 = \sum^m _{i=1} (dx_i)^2 - \sum^{m+n}_{\a=m+1}(dx_\a)^2.
\end{align*}
For $a=\left(a_1,\dotsc, a_{m+n}\right)\in\mathbb{R}^{m+n}_n$ and $b=\left(b_1,\dotsc, b_{m+n}\right)\in\mathbb{R}^{m+n}_n$, introduce
\begin{align*}
\left\langle a, b\right\rangle:=\sum_{i=1}^m a_ib_i-\sum_{\alpha=m+1}^{m+n}a_{\alpha}b_{\alpha},\quad |a|^2:=\left\langle a, a\right\rangle,\quad\|a\|:=\sqrt{|\left\langle a, a\right\rangle|}.
\end{align*}
An $m$-dimensional submanifold $M^m$ in $\mathbb{R}^{m+n}_n$ is called spacelike if the induced metric on $M^m$ is a Riemannian metric.
The mean curvature flow (MCF) in the pseudo-Euclidean space  is a one-parameter family of immersions
$X_t=X(\cdot, t): M^m \rightarrow  \mathbb{R}^{m+n}_n$  with the
corresponding image $M_t=X_t(M^m)$ such that
\begin{align}\label{ODE-add}
    \begin{cases}
    \dfrac{\partial}{\partial t}X(x,t)=H(x,t),&(x,t)\in M^m\times[0,T);\\
    X(x,0)=X(x),&x\in M^m,
    \end{cases}
\end{align}
is satisfied, here $H(x, t)$ is the mean curvature vector of $M_t$
at $X(x, t)$ in  $\mathbb{R}^{m+n}_n$. There are many interesting and essential results on the mean curvature flow  of spacelike submanifolds  in certain Lorentzian manifolds (see e.g. \cite{Eck1, Eck2, Eck3, Hal, Huang, Xin2}).

Let us firstly recall some facts in Euclidean spaces, Chern \cite{Chern} showed that entire graphs of constant mean curvature (CMC) in $\mathbb{R}^{m+1}$ are minimal. It is well known that these graphs must be hyperplanes for $m\leq 7$ (see Bernstein \cite{Bernstein} for $m=2$, De Giorgi \cite{Giorgi} for $m=3$, Almgren \cite{Almgren} for $m=4$ and Simons \cite{Simons} for $m\leq7$) and there are counterexamples for $m>7$ (see Bombieri-De Giorgi-Giusti \cite{BGG}). 

In Minkowski space $\mathbb{R}^{m+1}_1$, Calabi \cite{Cal} proposed the Bernstein problem for spacelike maximal hypersurfaces and proved that such hypersurfaces have to be hyperplanes when $m\leq 4$. Cheng-Yau \cite{CY} solved the problem for all $m$, in sharp contrast to the situation of the Euclidean space. Later, Ishihara \cite{Ish88} and Jost-Xin \cite{JX} generalized the results to higher codimension. The rigidity problem for spacelike submanifolds with parallel mean curvature was studied in \cite{JX, Xin91, XY97}.

On the other hand, by the work of Colding-Minicozzi \cite{CM} (see also \cite{Ang}), we know that, in the Euclidean space, minimal submanifolds and self-shrinkers share many geometric properties. Recall that $M^m$ is said to be a self-shrinker in $\mathbb{R}^{m+n}_n$ if
\begin{equation}\label{Def}
H=-\frac{1}{2}X^N,
\end{equation} which is an
important class of solutions to \eqref{ODE-add}, where $X^N$ is the normal part of $X$.  So it is natural to consider the rigidity of spacelike self-shrinkers in the pseudo-Euclidean space. Under the global conditions of Lagrangian entire graph or complete with the induced metric, there are plenty of related works, see e.g. \cite{Adames, CCY, CJQ2, DX, HW, LX}. It should mention that Chen-Qiu \cite{CQ} proved that only the affine planes are the  complete $m$-dimensional spacelike self-shrinkers in the pseudo-Euclidean space $\mathbb{R}^{m+n}_n$.

In this paper, we further study the geometry of the $m$-dimensional spacelike entire self-shrinking graphs in $\mathbb{R}^{m+n}_n$. 
By establishing a new volume growth estimate  (see Theorem \ref{thm-volume}) for the spacelike entire graphs  and the Co-Area formula (see Federer \cite{Federer} for Lipschitz functions or Fleming-Rishel \cite{FR} for BV functions), we give various growth estimates on the mean curvature and the $w$-function when the spacelike self-shrinking graph is not a linear subspace, these lead to rigidity results if the growth conditions are not satisfied.
\begin{thm}\label{thm-C1}

Let $X: M^m \to \mathbb{R}^{m+n}_n$ be a spacelike entire self-shrinking graph. Assume that
the origin $o\in M^m$ and $M^m$ is not a linear subspace. Then  the mean curvature satisfies 
\beq\label{eqn-H0}
\limsup_{R\to\infty}\dfrac{R^2}{\log\left(\int_{D_R}\|H\|^4e^{-\frac{z}{4}}\right)}\leq4\sqrt{m},
\eeq
where $D_R:= M^m \cap \{ p \in \mathbb{R}^{m+n}_n: z(p) \leq R^2  \}$ and $z=|X|^2$.
\end{thm}

\begin{rem}

Clearly, Theorem \ref{thm-C1} implies a rigidity result for the  spacelike entire self-shrinking graph if 
\[
 \limsup_{R\to\infty}\dfrac{R^2}{\log\left(\int_{D_R}\|H\|^4e^{-\frac{z}{4}}\right)}>4\sqrt{m}.
\]
In particular, by the above Theorem \ref{thm-C1} and Theorem \ref{thm-volume} which is stated in section 3, if $\|H\|^2\leq Ce^{\alpha z}$ for some constant $C>0$ and  $\alpha<\frac{1}{8}$, then the  spacelike entire self-shrinking graph has to be a linear subspace (see also  Theorem 1.1 in \cite{LX} or Theorem 2.1 in \cite{LQ}). Moreover, the growth condition can be weakened as  $\|H\|^2\leq Ce^{\alpha z}$ for $\alpha<\frac16+\frac{1}{6\sqrt{m}}$, see Corollary \ref{cor-R0} in section 5.
\end{rem}

\begin{thm}\label{thm-RW}

Let $X: M^m \to \mathbb{R}^{m+n}_n$ be a spacelike  entire self-shrinking graph. Assume that 
the origin $o\in M^m$ and $M^m$ is not a linear subspace. Then  the $w$-function satisfies 
\beq\label{eqn-H10}
\limsup_{R\to \infty} \frac{(\log R)^2}{\int_{D_R} w^2 (\log w)^2 e^{-\frac{z}{4}}} < \infty.
\eeq
Here the definition of $w$-function is given in Section 2. 
\end{thm}

\begin{rem}
Ding-Wang \cite{DW}  showed that the spacelike entire self-shrinking graph satisfying $\lim\limits_{|x|\to\infty}\frac{\log\det(g_{ij}(x))}{|x|}=0$ is a linear subspace (see Theorem 3 in \cite{DW}). By using the above Theorem \ref{thm-RW}, we can improve their result, see the details in the proof of Corollary \ref{cor-DW} in section 5.
\end{rem}

The article will be organized as follows. In the next section, we shall give some preliminaries. In Section 3, we establish  a new volume growth estimate for spacelike entire self-shrinking graphs. 
 Subsequently, in Section 4, we give the proof of Theorem \ref{thm-C1} and Theorem \ref{thm-RW}. Finally, as applications, various rigidity results for the spacelike self-shrinkers are presented in Section 5.

\section{Preliminaries}

Let $M^m$ be an $m$-dimensional spacelike submanifold in $\mathbb{R}^{m+n}_n$. The second fundamental form $B$ of $M^m$ in $\mathbb{R}^{m+n}_n$ is defined by
\begin{align*}
  B_{UW}:= \left(\overline{\n}_U W\right)^N  
\end{align*}
for $U, W \in \Gamma(TM^m)$. We use the notation $( \cdot )^T$ and $(
\cdot )^N$ for the orthogonal  projections into the tangent bundle
$TM^m$ and the normal bundle $NM^m$, respectively. For $\nu \in
\Gamma(NM^m)$ we define the shape operator $A^\nu: TM^m \rightarrow TM^m$
by
\begin{align*}
    A^\nu (U):= -\left(\overline{\n}_U \nu\right)^T.
\end{align*}
We have the following 
\begin{align*}
    \left\langle A^{\nu}(U),W\right\rangle=\left\langle A^{\nu}(W),U\right\rangle=\left\langle B_{UW},\nu\right\rangle.
\end{align*}
Taking the trace of $B$ gives the mean curvature vector $H$ of $M^m$
in $\mathbb{R}^{m+n}_{n}$ and
\begin{align*}
   H:= \hbox{trace} (B) = \sum_{i=1}^m B_{e_i e_i}, 
\end{align*}
where $\{ e_i \}$ is a local orthonormal frame field of $TM^m$. 
The Gauss equation, Cadazzi equation and Ricci equation are (cf. \cite{Xin19})
\begin{align*}
    R_{ijkl}=&\left\langle B_{e_ie_k},B_{e_je_l}\right\rangle-\left\langle B_{e_ie_l},B_{e_je_k}\right\rangle,\\
    \left(\nabla_{e_i}B\right)_{e_je_k}=&\left(\nabla_{e_j}B\right)_{e_ie_k},\\
    R\left(e_i,e_j,\nu,\mu\right)=&\left\langle A^{\nu}\left(e_i\right),A^{\mu}\left(e_j\right)\right\rangle-\left\langle A^{\nu}\left(e_j\right),A^{\mu}\left(e_i\right)\right\rangle.
\end{align*}

\vskip12pt

All spacelike $m$-planes (oriented $m$-subspaces) in $\mathbb{R}^{m+n}_n$ form the pseudo-Grassmannian manifold $G^{n}_{m, n}$. It is a specific Cartan-Hadamard manifold which is the noncompact dual space of the Grassmannian manifold $G_{m, n}$.

Let $P_1, P_2 \in G^{n}_{m, n}$ be two spacelike $m$-planes in $\mathbb{R}^{m+n}_n$. The angles between $P_1$ and $P_2$ are defined by the critical values of angel $\theta$ between a nonzero vector $x$ in $P_1$ and its orthogonal projection $x^*$ in $P_2$ as $x$ runs through $P_1$.

Assume that $e_1, ..., e_m$ are oriented  orthonormal vectors which span $P_1$ and $a_1, ..., a_m$ for $P_2$. For a nonzero vector in $P_1$, 
\[
x = \sum_{i} x_i e_i,
\]
its orthonormal projection in $P_2$ is 
\[
x^* = \sum_i x_{i}^* a_i.
\]
Hence for any $y \in P_2$, we obtain 
\[
\la x-x^*, y \ra = 0.
\]
Let $W_{ij} = \la e_i, a_j \ra$. Then we get 
\[
x_j^{*} = \sum_i W_{ij}x_i.
\]
A direct computation yields
\begin{align*}
    \left\langle x, x^*\right\rangle=|x^*|^2=\sum_{i,j,k}x_iW_{ij}W_{kj}x_k.
\end{align*}
Since $WW^T$ is symmetric, so we can choose appropriate orthonormal vectors $\{e_1, ..., e_m\}$, such that  $WW^T = \mathrm{diag}\{ \mu_{1}^2, ..., \mu_{m}^2 \}$ with $\mu_i = \cosh \theta_i\geq1$. Hence
\begin{align*}
    \left\langle x, x^*\right\rangle\geq|x||x^*|.
\end{align*}
The angle $\theta$ between $x$ and $x^*$ is defined by 
\[
\cosh \theta = \frac{\la x, x^* \ra}{|x||x^*|}.
\]

For the spacelike $m$-submanifold of $\mathbb{R}^{m+n}_n$, let $\{e_i\}$ be a local orthonormal frame of $TM^m$ such that $e_1\wedge e_2\wedge\dotsm\wedge e_m$ gives the orientation of $M^m$.
For the fixed $P_2 \in G^n_{m, n}$, which is spanned by the oriented orthonormal basis $a_1,\dotsc,a_m$, define the $w$-function as follows
\[
w=\left\langle e_1\wedge e_2\wedge\dotsm\wedge e_m,a_1\wedge a_2\wedge \dotsm\wedge a_m\right\rangle = \det W.
\]
Then, up to multiplying by -1, the $w$-function given by the spacelike $m$-plane $P$ satisfies $w\geq1$ when restricted on $M^m$.   
Now we have
\[
w= \prod_i \cosh \theta_i = \prod_i \frac{1}{\sqrt{1-\l_i^{2}}}, \quad \l_i = \tanh \theta_i.
\]
Choose timelike vectors $a_{m+\a}$ such that $\{ a_i, a_{m+\a}| i=1,...,m; \a = 1,...,n \}$ is an orientated orthonormal Lorentzian basis of $\mathbb{R}^{m+n}_n$. Then we can choose appropriate $\{ a_i, a_{m+\a}| i=1,...,m; \a = 1,...,n \}$ such that
\begin{align*}
   \{ e_i = \cosh \theta_i a_i + \sinh \theta_i a_{m+i}| i=1,\dotsc,m\}
\end{align*}
is an orientated tangent orthonormal basis of $M^m$, here $\theta_i = 0$ for $i>\min\{m, n\}$. 

\section{Volume growth estimate}

We derive the following volume growth estimate for the spacelike entire graphs in pseudo-Euclidean space $\mathbb{R}^{m+n}_n$. 


\begin{thm}\label{thm-volume}
Let $X: M^m \longrightarrow \mathbb{R}^{m+n}_{n}$ be an $m$-dimensional  spacelike entire graph. Let $z=\la X, X \ra$.
Assume the origin $o\in M^m$, then 
\begin{align}\label{eq:volume}
    \limsup_{R\to\infty}R^{-2m}\int_{\{z\leq R^2\}\cap M^m}w<\infty.
\end{align}
Consequently, for every $\alpha>0$, 
\begin{align*}
    \int_{M^m}we^{-\alpha z}<\infty.
\end{align*}
\end{thm}

\begin{proof}

Since $M^m$ is an entire graph, namely, $M^m$ can be written as $ \{ \left. X=(x, u(x)) \right| x \in \mathbb{R}^m, u=(u^1, u^2, ..., u^n)  \}$. 
By using the  singular value decomposition (see \cite{Sal08}), by an action of $\mathrm{SO}(m)\times\mathrm{SO}(n)$ we can choose a new Lorentzian coordinates $\{x_1, \dotsc, x_m, x_{m+1},\dotsc, x_{m+n}\}$ on $\mathbb{R}^{m+n}_n$ such that at a considered point to be calculated,
\[
du\left(\frac{\p}{\p x_i}\right) = \l_i \frac{\p}{\p x_{m+i}}.
\]
Here $\l_i=0 $ for $i>\min\{m, n\}$. For simplicity, we denote $E_i=\frac{\p}{\p x_i}, E_{m+\a}=\frac{\p}{\p x_{m+\a}}$, ($i=1,...,m,$  and $\a= 1,..., n$). Since $M^m$ is spacelike, we have $|\l_i| < 1$. Let
\begin{equation*}
e_i:=\dfrac{1}{\sqrt{1-\l_i^{2}}}(E_i + \l_i E_{m+i}).
\end{equation*}
Define the $w$-function as 
\[
w=\la e_1\wedge ... \wedge e_m,  E_1\wedge... \wedge E_m\ra.
\]
Then we derive
\begin{equation}\label{eqn-w-function}
w= \prod_i \dfrac{1}{\sqrt{1-\l_i^{2}}}=\dfrac{1}{\sqrt{\det(g_{ij})}},
\end{equation}
where $g_{ij} =\d_{ij}-\sum_{\alpha} u^\a_{i}u^{\a}_j$ is the induced metric on $M^m$.
Moreover, $z=|x|_{\mathbb{R}^m}^2-|u(x)|_{\mathbb{R}^n}^2$. Since $M$ is spacelike and $u(0)=0$ we have $|u(x)|<|x|$ (for $|x|\neq 0$), thus there exists a constant $\d>0$, such that for each point $(x, u(x)) \in M$ with $|x|=\e$, we get $|y|\leq \d <\e$. Without loss of generality, assume $x\neq 0$, let $\bar{x}=\frac{\e}{|x|}x$ and $(\bar{x}, \bar{u}(\bar{x}))\in M$, then $|\bar{x}|=\e$ and $|\bar{u}(\bar{x})|\leq \d$. Again since $M$ is spacelike, we obtain 
\[
|u(x)-\bar{u}(\bar{x})|  \leq |x-\bar{x}|.
\]
It follows that 
\begin{equation*}\aligned
|u(x)|^2 \leq & |x|^2 + |\bar{x}|^2 -|\bar{u}(\bar{x})|^2 + 2\la u(x), \bar{u}(\bar{x}) \ra -2\la x, \bar{x} \ra \\
\leq & |x|^2 +\e^2 -|\bar{u}(\bar{x})|^2 +2|u(x)||\bar{u}(\bar{x})|- 2\e |x| \\
\leq & |x|^2 - 2(\e-\d)|x| +\e^2.
\endaligned
\end{equation*}
This implies that 
\[
|u(x)| \leq |x|+\e
\]
and 
\begin{equation}\label{eqn-pseudo}
|x| \leq \frac{z+\e^2}{2(\e-\d)} =: C_1 (z+1),
\end{equation}
where $C_1$ is a positive constant depending only on $\e$ and $\d$. Direct computation gives us
\begin{equation*}\aligned
|X^T|^2 = &\la X, e_i \ra^2 =\sum_i \frac{1}{1-\l_i^{2}}\left( \la X, E_i \ra + \l_i \la X, E_{m+i} \ra \right)^2 \\
\leq & \sum_i \frac{2}{1-\l_i^{2}} \left( \la X, E_i \ra^2 + \la X, E_{m+i} \ra^2 \right)\\
\leq &  2 w^2 (|x|^2 + |u(x)|^2)
\endaligned
\end{equation*}
 Therefore we obtain
\begin{equation}\label{eqn-position}
|X^T| \leq C_2 w (z+1)
\end{equation}
here $C_2$ is a positive constant depending only on $\e$ and $\d$.

Since $M$ is entire, by (\ref{eqn-pseudo}), $z$ is proper. Hence, for every $R>0$,
\begin{align*}
    \int_{\{z\leq R^2\}\cap M^m}w=\int_{\{x\in\mathbb{R}^m:|x|_{\mathbb{R}^m}^2-|u(x)|_{\mathbb{R}^n}^2\leq R^2\}}w\sqrt{\det g}dx\leq\int_{\{|x|_{\mathbb{R}^m}\leq C_1\left(R^2+1\right)\}}dx=C_1^m\left(R^2+1\right)^{m}\int_{\{|x|_{\mathbb{R}^m}\leq1\}}dx,
\end{align*}
which gives the desired estimate \eqref{eq:volume}. 

Since $|\nabla\sqrt{z}|=\frac{|X^T|}{\sqrt{z}}\geq1$ whenever $X\neq o$, by  the Co-Area formula and integration by parts, we obtain
\begin{align*}
  \int_{M^m}we^{-\alpha z}=&\int_0^{\infty}\left(\int_{\{\sqrt{z}= R\}\cap M^m}\dfrac{1}{|\nabla\sqrt{z}|}we^{-\alpha z} \right)dR \\
  =&\int_0^{\infty}e^{-\alpha R^2}d\int_{\{\sqrt{z}\leq R\}\cap M^m}w\\
  =& 2\alpha\int_{0}^{\infty}Re^{-\alpha R^2}\left(\int_{\{\sqrt{z}\leq R\}\cap M^m}w\right)d R\\
  \leq& 2 \alpha\int_{0}^{1}Re^{-\alpha R^2}\left(\int_{\{\sqrt{z}\leq R\}\cap M^m}w\right)d R+C\int_{1}^{\infty}Re^{-\alpha R^2}R^{2m}d R\\
  <&\infty.
\end{align*}
\end{proof}

\section{Proof of Theorem \ref{thm-C1} and Theorem \ref{thm-RW}}

Let $V:= -\frac{1}{2}X^T $ and $ \D_V:= \D +\la V, \n\cdot \ra$ be the drift-Laplacian.

\begin{proof}[Proof of Theorem \ref{thm-C1}]
Let $\{e_1, ..., e_m\}$ be a local tangent orthonormal frame field on $M^m$ such that $\n_{e_i}e_j = 0$ at a  considered point to be calculated. From the self-shrinker equation \eqref{Def}, we obtain
\beq\label{eqn-H1}
\n_{e_j} H = -\dfrac{1}{2} \n_{e_j} \left(X- \la X, e_k \ra e_k\right)^N = \dfrac{1}{2} \la X,  e_k \ra B_{jk}
\eeq
and 
\[
\n_{e_i}\n_{e_j} H = \dfrac{1}{2} B_{ij} - \la H, B_{ik} \ra B_{jk} + \dfrac{1}{2} \la X, e_k \ra \n_{e_i} B_{jk}.
\]
Then using the Codazzi equation, we derive
\begin{align*}
\D_V |H|^2 = & \D |H|^2 + \left\la V, \n |H|^2 \right\ra  \\
=& 2\la \n_{e_i}\n_{e_i} H, H \ra + 2|\n H|^2 + \left\la V, \n |H|^2 \right\ra \\
=& |H|^2 - 2 \la H, B_{ik} \ra^2 + \frac{1}{2}\n_{X^T} |H|^2 + 2|\n H|^2 - \dfrac{1}{2}\left\la X^T, \n |H|^2  \right\ra \\
=&   |H|^2 - 2 |A^H|^2 + 2|\n H|^2.
\end{align*}
 It follows that
\begin{equation}\label{eqn-H2}\aligned
\D_V \|H\|^2 =   \|H\|^2 + 2 |A^H|^2 + 2\|\n H\|^2 \geq  2 \left|A^H\right|^2,
\endaligned
\end{equation}
here $\|H\|^2$ is the absolute value of the square of the mean curvature vector $H$. 

By \eqref{eqn-H1}, we get
\begin{equation*}\aligned
\n_{X^T} \|H\|^2 = &  -2 \left\la \n_{X^T} H, H \right\ra \\
=&-2\left\la X, e_j \right\ra \left\la \n_{e_j} H, H \right\ra \\
= & -2\la X, e_j \ra \left\la \frac{1}{2}\la X, e_k \ra B_{jk}, H \right\ra \\
=& - \left\la B(X^T, X^T), H \right\ra \\
=& - \left\la A^H(X^T),  X^T \right\ra.
\endaligned
\end{equation*}
Note that $X=X^T + X^N$, therefore we have 
\[
z= \la X, X \ra =\left |X^T\right|^2 + \left|X^N\right|^2 = \left|X^T\right|^2 -\left \|X^N\right\|^2,
\]
where $\left \|X^N\right\|^2$ is the absolute value of the square of the timelike vector $X^N$. Then by the self-shrinker equation \eqref{Def}, we obtain
\[
\left|X^T\right|^2 = z + 4\|H\|^2.
\]
Denote $B_R := \left\{ p\in \mathbb{R}^{m+n}_n: z(p)\leq R^2 \right\}$ and $D_R:= M^m\cap B_R$.  By (\ref{eqn-pseudo}), $z$ is proper, this implies that $D_R$ is compact in $M^m$. Thus a direct computation yields 
\begin{align*}
\int_{D_R}\D_V \|H\|^2 e^{-\frac{z}{4}} =  & \int_{D_R} e^{\frac{z}{4}}\div \left( e^{-\frac{z}{4}} \n \|H\|^2 \right) e^{-\frac{z}{4}} \\
=& \int_{D_R} \div \left( e^{-\frac{z}{4}}\n \|H\|^2 \right) \\
=& \int_{\p D_R} \left\la e^{-\frac{z}{4}}\n \|H\|^2, \frac{X^T}{|X^T|} \right\ra \\
=& \int_{\p D_R} \frac{1}{|X^T|} \n_{X^T} \|H\|^2 e^{-\frac{z}{4}}\\
 = & - \int_{\p D_R} \frac{\left\la A^H(X^T), X^T \right\ra}{\left|X^T\right|} e^{-\frac{z}{4}} \\
 \leq& \int_{\p D_R} \left|A^H\right| \left|X^T\right| e^{-\frac{z}{4}} \\
 \leq & \left( R \int_{\p D_R} \frac{\left|A^H\right|^2}{\left|X^T\right|} e^{-\frac{z}{4}}\right)^{\frac{1}{2}} \left( R^{-1} \int_{\p D_R} \frac{\left(z+4\|H\|^2\right)^2}{\left|X^T\right|} e^{-\frac{z}{4}} \right)^{\frac{1}{2}}.
\end{align*}
Namely
\begin{align}
    \label{eqn-H3}
\int_{D_R}\D_V \|H\|^2 e^{-\frac{z}{4}}\leq \left( R \int_{\p D_R} \frac{\left|A^H\right|^2}{\left|X^T\right|} e^{-\frac{z}{4}}\right)^{\frac{1}{2}} \left( R^{-1} \int_{\p D_R} \frac{\left(z+4\|H\|^2\right)^2}{\left|X^T\right|} e^{-\frac{z}{4}} \right)^{\frac{1}{2}}.
\end{align}
The Cauchy inequality implies
\begin{align}\label{eqn-A1}
   \left|A^H\right|^2 = \sum_{i, j} \left\la B_{ij}, H \right\ra^2 \geq \sum_i \left\la B_{ii}, H \right\ra^2 \geq \dfrac{\left( \sum_i \left\la B_{ii}, H \right\ra \right)^2}{m}  = \frac{\|H\|^4}{m}. 
\end{align}

By the assumption that $M^m$ is not a linear subspace, we can conclude that $M^m$ is not maximal, i.e.,  $H\not\equiv 0$. Otherwise, by the proof of Theorem 4.2 in \cite{JX}, we derive that $M$ is a linear subspace, this yields the contradiction.  Then there exists $R_0 > 0$, such that  for any $R > R_0$, 
\[
\int_{D_R} \left|A^H\right|^2 e^{-\frac{z}{4}} \geq \frac{1}{m} \int_{D_R} \|H\|^4 e^{-\frac{z}{4}} > 0.
\]

Let 
\[
F(R) := \int_{D_R} \left|A^H\right|^2 e^{-\frac{z}{4}}, \quad G(R):= \int_{D_R} \left(z+4\|H\|^2\right)^2 e^{-\frac{z}{4}}.
\]
By the Co-Area formula, we have
\begin{equation*}\aligned
F(R) = & \int_{0}^R \int_{\p D_R} \frac{\left|A^H\right|^2 e^{-\frac{z}{4}}}{|\n \sqrt{z}|} = \int_{0}^R \left( r e^{-\frac{r^2}{4}} \int_{\p D_r} \frac{\left|A^H\right|^2}{\left|X^T\right|} \right)dr, \\
G(R) = & \int_{0}^R \int_{\p D_r}\frac{\left(z+4\|H\|^2\right)^2 e^{-\frac{z}{4}}}{|\n \sqrt{z}|} = \int_{0}^R   \left( r e^{-\frac{r^2}{4}} \int_{\p D_r} \frac{\left(z+4\|H\|^2\right)^2 }{\left|X^T\right|} \right)dr.
\endaligned
\end{equation*}
It follows that
\begin{equation*}\aligned
F'(R) =   R e^{-\frac{R^2}{4}} \int_{\p D_R} \dfrac{\left|A^H\right|^2}{\left|X^T\right|}, \quad
G'(R) =  R e^{-\frac{R^2}{4}} \int_{\p D_R} \dfrac{\left(z+4\|H\|^2\right)^2 }{\left|X^T\right|}.
\endaligned
\end{equation*}
From \eqref{eqn-H2} and \eqref{eqn-H3}, we obtain 
\begin{equation*}\aligned
4F(R)^2 \leq F'(R) \cdot R^{-2} G'(R).
\endaligned
\end{equation*}
Namely, 
\[
\dfrac{R^2}{G'(R)} \leq \dfrac{F'(R)}{4F(R)^2} = -\dfrac{1}{4} \left( \dfrac{1}{F(R)} \right)',  \quad \forall R > R_0. 
\]
Therefore for any fixed $r$ satisfying $R> r > R_0$, 
\begin{align*}
\dfrac{1}{4} (R^2 -r^2)^2 = & \left( \int_r^{R} sds \right)^2\leq \int_{r}^R \frac{s^2}{G'(s)} ds \cdot \int_r^{R} G'(s)ds \leq -\dfrac{1}{4} \left( \dfrac{1}{F(R)} - \dfrac{1}{F(r)} \right) \cdot \left( G(R) - G(r) \right),
\end{align*}
which gives
\begin{align}\label{eq:refine0}
  (R^2 -r^2)^2\leq& \dfrac{G(R)}{F(r)}  
\end{align}

We claim that
\begin{align}\label{eqn-AH}
\int_{M^m}\left|A^H\right|^2e^{-\frac{z}{4}} = \infty. 
\end{align}
In fact, let $R$ go to infinity and then $r$ go to infinity,
\begin{equation*}\aligned
\limsup_{R\to \infty} \dfrac{R^4}{G(R)} \leq \liminf_{r\to\infty}\dfrac{1}{F(r)}=\dfrac{1}{\int_{M^m}|A^{H}|^2e^{-\frac{z}{4}}}\leq \dfrac{m}{\int_{M^m}\|H\|^4e^{-\frac{z}{4}}}.
\endaligned
\end{equation*}
This implies that
\begin{equation}\label{eqn-A2}\aligned
\limsup_{R\to \infty} \dfrac{R^4}{ \int_{D_R } 2z^2 e^{-\frac{z}{4}} + \int_{D_R} 32\|H\|^4 e^{-\frac{z}{4}}} \leq \limsup_{R\to \infty} \dfrac{R^4}{G(R)} \leq \dfrac{m}{\int_{M^m}\|H\|^4e^{-\frac{z}{4}}}.
\endaligned
\end{equation}
By Theorem \ref{thm-volume}, $\int_{M^m} e^{\a z} < \infty $ for any $\a < 0$,  so we can conclude that 
\begin{align}\label{eqn-zweight}
    \int_{D_R } z^2 e^{-\frac{z}{4}} \leq  \int_{M^m } z^2 e^{-\frac{z}{4}} < \infty.
\end{align}
Thus from the inequality \eqref{eqn-A2}, we get
\begin{equation}\label{eqn-H4}\aligned
\limsup_{R\to \infty} \frac{R^4}{  \int_{D_R} 32\|H\|^4 e^{-\frac{z}{4}}} \leq \dfrac{m}{\int_{M^m}\|H\|^4e^{-\frac{z}{4}}}.
\endaligned
\end{equation}
If  $\int_{M^m}\|H\|^4e^{-\frac{z}{4}} < \infty$, then by \eqref{eqn-H4}, we conclude that
\begin{align*}
    \infty > \dfrac{m}{\int_{M^m}\|H\|^4e^{-\frac{z}{4}}}\geq\limsup_{R\to \infty} \dfrac{R^4}{  \int_{D_R} 32\|H\|^4 e^{-\frac{z}{4}}}\geq\limsup_{R\to \infty} \dfrac{R^4}{  \int_{M^m} 32\|H\|^4 e^{-\frac{z}{4}}}=\infty.
\end{align*}
This yields the contradiction. Hence 
\begin{align}\label{eqn-AH0}
\int_{M^m}\|H\|^4e^{-\frac{z}{4}} = \infty. 
\end{align}
Then \eqref{eqn-AH} follows from \eqref{eqn-A1}.

 From \eqref{eqn-H2}, \eqref{eqn-H3} and \eqref{eqn-A1}, we get
 \begin{align}
     2\int_{D_R}\left|A^{H}\right|^2e^{-\frac{z}{4}}\leq & \int_{D_R} \D_V\|H\|^2 e^{-\frac{z}{4}} \notag\\
     \leq &  \left( R \int_{\p D_R} \dfrac{\left|A^H\right|^2}{\left|X^T\right|} e^{-\frac{z}{4}}\right)^{\frac{1}{2}} \left( R^{-1} \int_{\p D_R} \dfrac{\left(z+4\sqrt{m}\left|A^H\right|\right)^2}{\left|X^T\right|} e^{-\frac{z}{4}} \right)^{\frac{1}{2}} \notag\\
     \leq&4\sqrt{m}\int_{\p D_R} \dfrac{\left(\frac{z}{4\sqrt{m}}+\left|A^H\right|\right)^2}{\left|X^T\right|} e^{-\frac{z}{4}}. \label{eqn-AH2}
 \end{align}
 Set
 \begin{align*}
     \hat F(R)=\int_{D_R}\left(\left|A^H\right|+\dfrac{z}{4\sqrt{m}}\right)^2e^{-\frac{z}{4}}.
 \end{align*}
$\forall \e \in \left(0, \frac12\right)$, for given positive constant $\d < \frac{\e}{1-\e}$,
\begin{equation}\aligned\label{eqn-Fhat1}
\hat{F}(R) \leq &  (1+\d)\int_{D_R}\left|A^H\right|^2 e^{-\frac{z}{4}} + \left( 1+ \dfrac{1}{\d} \right) \cdot \dfrac{1}{16m} \int_{D_R} z^2 e^{-\frac{z}{4}} \\
\leq & (1+\d)\int_{D_R}\left|A^H\right|^2 e^{-\frac{z}{4}} + \left( 1+ \dfrac{1}{\d} \right) \cdot \dfrac{1}{16m} \int_{M^m} z^2 e^{-\frac{z}{4}}. 
\endaligned
\end{equation}
Since $\d< \frac{\e}{1-\e}$, we get $\frac{1}{1-\e}-(1+\d)>0$. By \eqref{eqn-AH}, we obtain
\[
\lim_{R\to \infty } \int_{D_R} \left|A^H\right|^2 e^{-\frac{z}{4}} = \infty. 
\]
Note that $\int_{M^m} z^2 e^{-\frac{z}{4}}< \infty$ by \eqref{eqn-zweight}. Thus there exists $R_1>0$, such that when $R>R_1$, we have
\begin{align}\label{eqn-Fhat2}
\int_{D_R} \left|A^H\right|^2 e^{-\frac{z}{4}} > \frac{1}{\frac{1}{1-\e}-(1+\d)}\left( 1+ \dfrac{1}{\d} \right) \cdot \dfrac{1}{16m} \int_{M^m} z^2 e^{-\frac{z}{4}}.
\end{align}
Combining \eqref{eqn-Fhat1} with \eqref{eqn-Fhat2}, it follows 
\begin{equation}\label{eqn-Fhat3}
\hat{F}(R) < \dfrac{1}{1-\e} \int_{D_R}\left|A^H\right|^2 e^{-\frac{z}{4}}.
\end{equation}
By the Co-Area formula, we obtain
\begin{equation*}\aligned
\hat{F}(R) = \int_{D_R}\left(\left|A^H\right|+\dfrac{z}{4\sqrt{m}}\right)^2e^{-\frac{z}{4}} = \int_{0}^R \left(r \int_{\p D_r} \dfrac{ \left( \left|A^H\right|+\frac{z}{4\sqrt{m}} \right)^2}{\left|X^T\right|}e^{-\frac{z}{4}}\right) dr. 
\endaligned
\end{equation*}
This implies that 
\begin{equation}\label{eqn-Fhat4}
\hat{F}'(R) = R \int_{\p D_R} \dfrac{ \left( \left|A^H\right|+\frac{z}{4\sqrt{m}} \right)^2}{\left|X^T\right|}e^{-\frac{z}{4}}.
\end{equation}
 Thus from \eqref{eqn-AH2}, \eqref{eqn-Fhat3} and \eqref{eqn-Fhat4}, we get
 \begin{align*}
     0<\dfrac{1-\varepsilon}{2\sqrt{m}}\hat F(R)\leq\dfrac{1}{R}\hat F'(R),
 \end{align*}
which implies for some $R_2>R_1$,
 \begin{align*}
     \dfrac{1-\varepsilon}{4\sqrt{m}}\left(R^2-R_2^2\right)\leq\log\dfrac{\hat F(R)}{\hat F(R_2)},\quad\forall R>R_2.
 \end{align*}
 Namely,
 \[
  \int_{D_R}\left|A^H\right|^2e^{-\frac{z}{4}} > (1-\e)\dfrac{\hat{F}(R_2)}{\exp\left( \frac{1-\e}{4\sqrt{m}}R_{2}^2 \right)}\cdot \exp \left( \dfrac{1-\e}{4\sqrt{m}}R^2 \right)
 \]
Thus for $R$ sufficiently large, we have
 \begin{align}\label{eqn-AH3}
     \int_{D_R}\left|A^H\right|^2e^{-\frac{z}{4}}\geq\exp\left(\dfrac{1-2\varepsilon}{4\sqrt{m}}R^2\right).
 \end{align}
 According to \eqref{eq:refine0} and \eqref{eqn-AH0}, as the similar reason to derive \eqref{eqn-Fhat3},  for sufficiently large $R$, we obtain
 \begin{align*}
\left(R^2-r^2\right)^2 \leq\dfrac{\int_{D_R}\left(z+4\|H\|^2\right)^2e^{-\frac{z}{4}}}{\int_{D_r} \left|A^H\right|^2e^{-\frac{z}{4}}} 
\leq  \dfrac{16(1+\varepsilon)\int_{D_R}\|H\|^4e^{-\frac{z}{4}}+ C_{\varepsilon}}{\int_{D_r} \left|A^H\right|^2e^{-\frac{z}{4}}} 
\leq  \dfrac{16(1+2\varepsilon)\int_{D_R}\|H\|^4e^{-\frac{z}{4}}}{\int_{D_r} \left|A^H\right|^2e^{-\frac{z}{4}}}. 
\end{align*}
 Choosing $r = (1-\e)R$, by \eqref{eqn-AH3},
  \begin{equation*}\aligned
\left(2\e-\e^2\right)R^4 \leq  \dfrac{16(1+2\varepsilon)\int_{D_R}\|H\|^4e^{-\frac{z}{4}}}{\int_{D_{(1-\e)R}} \left|A^H\right|^2e^{-\frac{z}{4}}} \leq  \dfrac{16\left(1+2\varepsilon\right)\int_{D_R}\|H\|^4e^{-\frac{z}{4}}}{\exp\left(\dfrac{1-2\varepsilon}{4\sqrt{m}}(1-\e)^2R^2\right)}.
\endaligned
\end{equation*}
 Direct computation gives us 
  \begin{equation*}\aligned
\frac{1}{R^2}\left( \log(2\e-\e^2) + 4\log R \right) + \frac{(1-2\e)(1-\e)^2}{4\sqrt{m}}\leq \frac{1}{R^2}\log 16(1+2\varepsilon) + \frac{1}{R^2}\log \int_{D_R} \|H\|^4 e^{-\frac{z}{4}}.
\endaligned
\end{equation*}
Letting $R \to \infty$ in the above equality, we get
 \begin{align*}
     \limsup_{R\to\infty}\dfrac{R^2}{\log\left(\int_{D_R}\|H\|^4e^{-\frac{z}{4}}\right)}\leq\dfrac{4\sqrt{m}}{\left(1-2\varepsilon\right)\left(1-\varepsilon \right)^2}.
 \end{align*}
 Let $\varepsilon$ go to zero, we conclude that
 \begin{align*}
     \limsup_{R\to\infty}\dfrac{R^2}{\log\left(\int_{D_R}\|H\|^4e^{-\frac{z}{4}}\right)}\leq4\sqrt{m}.
 \end{align*}
\end{proof}

 Following the idea of the proof of Theorem \ref{thm-C1}, we give the

\begin{proof}[Proof of Theorem \ref{thm-RW}]
Let $B_R := \left\{ p\in \mathbb{R}^{m+n}_n: z(p)\leq R^2 \right\}$ and $D_R:= M^m\cap B_R$. By (\ref{eqn-pseudo}), $z$ is proper, thus $D_R$ is compact in $M^m$.
 Integration by parts gives us
\begin{equation}\label{eqn-H9}\aligned
\int_{D_R} \D_V(\log w) \log w e^{-\frac{z}{4}} =& \int_{D_R} \div\left( e^{-\frac{z}{4}}\n \log w \right)\log w \\
 =& \int_{D_R} \div \left( e^{-\frac{z}{4}}(\n \log w) \log w  \right) - \int_{D_R} \la e^{-\frac{z}{4}}\n\log w, \n \log w \ra \\
 =& \int_{\p D_R} \left\la e^{-\frac{z}{4}}(\n\log w) \log w, \frac{X^T}{|X^T|} \right\ra - \int_{D_R} |\n\log w|^2 e^{-\frac{z}{4}}.
\endaligned
\end{equation}
By Proposition 3.1 in \cite{LX}, we get
\begin{equation}\label{eqn-H8}
\D_V (\log w) \geq \frac{\|B\|^2}{w^2},
\end{equation}
here $\|B\|^2$ is the absolute value of the square of the second fundamental form.

 From \eqref{eqn-H9} and \eqref{eqn-H8}, we have
\begin{equation}\label{eqn-H5}\aligned
\int_{D_R} \frac{\|B\|^2}{w^2} \log w e^{-\frac{z}{4}} + \int_{D_R} |\n\log w|^2 e^{-\frac{z}{4}} \leq \int_{\p D_R} \left\la e^{-\frac{z}{4}}(\n\log w) \log w, \frac{X^T}{|X^T|}\right \ra.
\endaligned
\end{equation}

 Applying the  Cauchy-Schwarz inequality to the right hand side of \eqref{eqn-H5}, and using \eqref{eqn-position}, we get
\begin{equation}\label{eqn-H6}\aligned
\int_{\p D_R} \left\la e^{-\frac{z}{4}}(\n\log w) \log w, \frac{X^T}{|X^T|} \right\ra
\leq& \int_{\p D_R} |\n \log w| \log w e^{-\frac{z}{4}} \\
\leq & \left( \int_{\p D_R} \frac{R|\n \log w|^2}{|X^T|} e^{-\frac{z}{4}} \right)^{\frac{1}{2}}\left( \int_{\p D_R} \frac{R^{-1}(\log w)^2|X^T|^2}{|X^T|} e^{-\frac{z}{4}} \right)^{\frac{1}{2}} \\
\leq & \left( \int_{\p D_R} \frac{R|\n \log w|^2}{|X^T|} e^{-\frac{z}{4}} \right)^{\frac{1}{2}}\left( \int_{\p D_R} \frac{R^{-1}(\log w)^2 2C_2^{2} w^2(z+1)^2}{|X^T|} e^{-\frac{z}{4}} \right)^{\frac{1}{2}}  \\
\leq & C_3  \left( \int_{\p D_R} \frac{R|\n \log w|^2}{|X^T|} e^{-\frac{z}{4}} \right)^{\frac{1}{2}}\left( \int_{\p D_R} \frac{R^{3}w^2(\log w)^2}{|X^T|} e^{-\frac{z}{4}} \right)^{\frac{1}{2}}. 
\endaligned
\end{equation}
Combining \eqref{eqn-H5} with \eqref{eqn-H6}, we have
\begin{equation}\label{eqn-H7}\aligned
 \int_{D_R} |\n\log w|^2 e^{-\frac{z}{4}} \leq C_3  \left( \int_{\p D_R} \frac{R|\n \log w|^2}{|X^T|} e^{-\frac{z}{4}} \right)^{\frac{1}{2}}\left( \int_{\p D_R} \frac{R^{3}w^2(\log w)^2}{|X^T|} e^{-\frac{z}{4}} \right)^{\frac{1}{2}}. 
\endaligned
\end{equation}

Let 
\[
\widetilde{F}(R) := \int_{D_R} |\n\log w|^2 e^{-\frac{z}{4}}, \quad \widetilde{G}(R):= \int_{D_R} w^2 (\log w)^2 e^{-\frac{z}{4}}.
\]
The Co-Area formula gives 
\begin{equation*}\aligned
\widetilde{F}(R) =&  \int_{0}^R \int_{\p D_r}\frac{|\n \log w|^2 e^{-\frac{z}{4}}}{|\n \sqrt{z}|} = \int_{0}^R \int_{\p D_r} \frac{r|\n \log w|^2 e^{-\frac{z}{4}}}{|X^T|} dr, \\
\widetilde{G}(R) =& \int_{0}^R \int_{\p D_r} \frac{w^2 (\log w)^2 e^{-\frac{z}{4}}}{|\n \sqrt{z}|} = \int_{0}^R \int_{\p D_r} \frac{rw^2(\log w)^2 e^{-\frac{z}{4}}}{|X^T|} dr.
\endaligned
\end{equation*}
Then we get
\begin{equation*}\aligned
\widetilde{F}'(R) =  Re^{-\frac{R^2}{4}}\int_{\p D_R} \frac{|\n \log w|^2}{|X^T|}, \quad 
\widetilde{G}'(R) = Re^{-\frac{R^2}{4}}\int_{\p D_R} \frac{w^2 ( \log w)^2}{|X^T|}.
\endaligned
\end{equation*}
By \eqref{eqn-H7}, 
\[
\widetilde{F}(R)^2 \leq C_3^2 \widetilde{F}'(R) \cdot R^2 \widetilde{G}'(R).
\]
That is 
\[
\frac{1}{R^2 \widetilde{G}'(R)} \leq C_3^2 \frac{\widetilde{F}'(R)}{(\widetilde{F}(R))^2} = -C_3^2 \left( \frac{1}{\widetilde{F}(R)} \right)', \quad \forall R>1.
\]
Hence for any fixed $r\in (1, R)$, we derive
\begin{equation*}\aligned
\left( \log \frac{R}{r} \right)^2 =& \left( \int_r^{R} \frac{1}{s} ds\right)^2 \\
\leq& \int_{r}^R \frac{1}{s^2 \widetilde{G}'(s)}ds \cdot \int_{r}^R \widetilde{G}'(s)ds \\
\leq & -C_3^2 \left( \frac{1}{\widetilde{F}(R)} - \frac{1}{\widetilde{F}(r)} \right) \cdot (\widetilde{G}(R)-\widetilde{G}(r)) \\
\leq& C_3^2 \frac{\widetilde{G}(R)}{\widetilde{F}(r)}.
\endaligned
\end{equation*}
Let $R$ go to infinity and then $r$ go to infinity, 
\[
\limsup_{R\to \infty} \frac{(\log R)^2}{\widetilde{G}(R)} \leq C_3^2 \liminf_{r\to\infty}\frac{1}{\widetilde{F}(r)}=\dfrac{C_3^2}{\int_{M^m}|\nabla\log w|^2e^{-\frac{z}{4}}}.
\]
Since $M^m$ is not a linear subspace, by \eqref{eqn-H8} we know that $w$ can not be a constant. Hence the right hand side of the above inequality is finite. It follows that
\[
\limsup_{R\to \infty} \frac{(\log R)^2}{\int_{D_R} w^2 (\log w)^2 e^{-\frac{z}{4}}} < \infty.
\]
\end{proof}

 \section{Rigidity results for spacelike self-shrinkers}
 
 In this section, we shall give various rigidity results for spacelike self-shrinkers which can be viewed as the applications of Theorems \ref{thm-C1} and Theorem \ref{thm-RW}. 
 
 \vspace{1em}
 
 By Theorem \ref{thm-C1},  we have
 
 \begin{cor}\label{cor-R6}

Let $X: M^m \to \mathbb{R}^{m+n}_n$ be an $m$-dimensional  spacelike entire self-shrinking graph. Assume that 
the origin $o\in M^m$ and $M^m$ is not a linear subspace. Then  the mean curvature satisfies 
\beq\label{eqn-R9}
\limsup_{R\to \infty} \frac{R^2}{\log\left(\int_{D_R}\|H\|^3w e^{-\frac{z}{4}}\right)}\leq 4\sqrt{m}.
\eeq

 \end{cor}

\begin{proof}
\eqref{eqn-position} implies that
  \begin{align}\label{eqn-R8}
    \|H\|^4\leq\dfrac12\|H\|^3|X^T|\leq C_2\|H\|^3w (z+1).
\end{align}
 Then the conclusion follows from  \eqref{eqn-H0} and \eqref{eqn-R8}.
\end{proof}

As a consequence of Corollary \ref{cor-R6}, we obtain

\begin{cor}\label{cor-R0}

Let $X: M^m \to \mathbb{R}^{m+n}_n$ be an $m$-dimensional spacelike entire self-shrinking graph. Assume that the origin $o\in M^m$. If the mean curvature satisfies $\|H\|^2 \leq C e^{\a z}$ for any $\a < \frac{1}{6}+\frac{1}{6\sqrt{m}}$, here $C$ is a positive constant. Then $M^m$ must be a linear subspace.

\end{cor}

\begin{proof}
Choose $\alpha_0<\frac{1}{6}$ such that $\alpha<\alpha_0+\frac{1}{6\sqrt{m}}$. 
Suppose that $M^m$ is not a linear subspace. The assumption implies that 
\[
\int_{D_R} \|H\|^3 we^{-\frac{z}{4}} \leq C^{3/2} e^{\frac{3\left(\alpha-\alpha_0\right)}{2}R^2}\int_{D_R} w e^{\left(\frac{3\a_0}{2}-\frac{1}{4}\right)z}.
\]
Since $\a_0<\frac{1}{6}$, we get $\frac{3\a_0}{2}-\frac{1}{4}<0$. Then by Theorem  \ref{thm-volume}, $\int_{M^m} w e^{\left(\frac{3\a_0}{2}-\frac{1}{4}\right)z}< \infty.$   Therefore 
\[
 \limsup_{R\to \infty} \frac{R^2}{\log\left(\int_{D_R}\|H\|^3w e^{-\frac{z}{4}}\right)} \geq\dfrac{2}{3\left(\alpha-\alpha_0\right)}>4\sqrt{m}.
\]
Comparing the above inequality with \eqref{eqn-R9}, we conclude that $M^m$ is a linear subspace.
\end{proof}

By using gradient estimates and Corollary \ref{cor-R0}, we derive

\begin{cor}\label{cor-R1}

Let $X: M^m \to \mathbb{R}^{m+n}_n$ be an $m$-dimensional spacelike entire self-shrinking graph. Assume that the origin $o\in M$ and the $w$-function satisfies
\begin{equation*}
\limsup_{x \to \infty} \frac{\log w}{z} < \dfrac{1}{12}+\dfrac{1}{12\sqrt{m}}.
\end{equation*}
Then $M^m$ has to be a linear subspace.

\end{cor}

\begin{proof}
Let $f(R):= \max_{\{z=R\}} w$. 
Then by \eqref{eqn-H8} and the maximum principle, $f(R)$ is nondecreasing in $R$.  From the assumption, there exists $R_0>0$, such that when $R>R_0$, we have
\[
f(R) < e^{\left(\frac{1}{12}+\frac{1}{12\sqrt{m}}\right)(1-\varepsilon)R} \quad {\rm for} \quad {\rm some} \quad \epsilon >0.
\]
Choosing $a^2 > 2 R_0$. Let $B_a := \{ p\in \mathbb{R}^{m+n}_n: z(p)\leq a^2 \}$ and $D_a:= M^m\cap B_a$.  By (\ref{eqn-pseudo}), $z$ is proper, this implies that $D_a$ is compact in $M^m$.
 Define $\Phi: D_a \to \mathbb{R}$ by 
\[
\Phi:= (a^2-z)^2 \|H\|^2.
\]
 As $\Phi|_{\p D_a}=0$, $\Phi$ achieves an absolute maximum in the interior of $D_a$, say $\Phi \leq \Phi(q)$,
for some $q$ inside $D_a$.  We may  assume $\|H\|(q)\neq 0$. Then 
\[
\n \Phi(q)=0,\quad \D_{V}\Phi (q) \leq 0.
\]
By direct computation, we have 
\begin{equation*}\aligned
\n \Phi=  -2(a^2-z)\|H\|^2 \n z +(a^2-z)^2 \n\|H\|^2,
\endaligned
\end{equation*}
\begin{equation*}\aligned
\D_{V} \Phi= &  2\|H\|^2 \cdot |\n z|^2 - 2(a^2-z)\|H\|^2 \cdot \D_V z - 4(a^2-z)\left\la \n z, \n\|H\|^2 \right\ra +(a^2-z)^2 \D_V\|H\|^2.
\endaligned
\end{equation*}
From $\n \Phi (q)= 0$, we get at $q$
\begin{equation}\label{eqn-R1}
\frac{\n\|H\|^2}{\|H\|^2} - \frac{2\n z}{a^2-z} =0
\end{equation}
And by $\D_{V}\Phi(q) \leq 0$, we obtain at $q$
\beq\label{eqn-R2}\aligned
-\frac{4\la\n z, \n \|H\|^2\ra}{(a^2-z)\|H\|^2} & + \frac{\D_{V}\|H\|^2}{\|H\|^2} + \frac{2|\n z|^2}{(a^2-z)^2} -\frac{2\D_{V}z}{a^2-z} \leq 0.
\endaligned
\eeq
Substituting  \eqref{eqn-R1} into \eqref{eqn-R2}, we get
\begin{equation}\label{eqn-R3}\aligned
\frac{\D_{V}\|H\|^2}{\|H\|^2}  -\frac{2\D_{V}z}{a^2-z} - \frac{6|\n z|^2}{(a^2-z)^2} \leq 0.
 \endaligned
\end{equation}
Direct computation gives us 
\[
\D_V z= 2m-z, \quad \n z = 2X^T.
\]
Combining \eqref{eqn-H2}, \eqref{eqn-A1} with \eqref{eqn-R3}, we derive
\begin{equation}\label{eqn-R4}\aligned
\|H\|^2 \leq & \frac{m}{2}\frac{\D_V\|H\|^2}{\|H\|^2} \leq \frac{m}{2}\left(\frac{2\D_{V}z}{a^2-z} + \frac{6|\n z|^2}{(a^2-z)^2} \right) 
\leq  \frac{m}{2}\left(\frac{4m}{a^2-z} + \frac{24|X^T|^2}{(a^2-z)^2} \right)
\endaligned
\end{equation}
By \eqref{eqn-position} and \eqref{eqn-R4}, we get
\begin{equation*}
\aligned
\|H\|^2 
\leq  \frac{m}{2}\left(\frac{4m}{a^2-z} + \frac{48C_2^{2}w^2(z+1)^2}{(a^2-z)^2} \right).
\endaligned
\end{equation*}
Hence for $\delta=\sqrt{1-\varepsilon^2}$
\begin{equation}\label{eqn-A3}\aligned
\max_{D_{\delta a}}\Phi\leq \Phi(q)\leq  \frac{m}{2}\left(4ma^2 + 48C_{2}^2(a^2+1)^2f(a^2)^2  \right).
\endaligned
\end{equation}
By the definition of $\Phi$ and \eqref{eqn-A3}, we conclude that some constant $C$ such that for $a$ sufficiently large we have
\begin{align*}
    \max_{D_{\delta a}}\|H\|^2\leq   \dfrac{C}{(1-\delta^2)^2}f(a^2)^2\leq\dfrac{C}{\varepsilon^4}e^{\left(\frac{1}{6}+\frac{1}{6\sqrt{m}}\right)(1-\varepsilon)a^2}=\dfrac{C}{\varepsilon^4}e^{\left(\frac{1}{6}+\frac{1}{6\sqrt{m}}\right)\frac{(\delta a)^2}{1+\varepsilon}}.
\end{align*}
Hence for every $q'\in M^m$ with $z(q')=\delta a$,
\begin{align*}
    \|H\|^2(q')\leq \dfrac{C}{\varepsilon^4}e^{\left(\frac{1}{6}+\frac{1}{6\sqrt{m}}\right)(1-\varepsilon)z(q')}.
\end{align*}
In other words, we get the following estimate
\begin{align*}
    \|H\|^2\leq C_{\varepsilon}e^{\left(\frac{1}{6}+\frac{1}{6\sqrt{m}}\right)(1-\varepsilon)z}.
\end{align*}
By Corollary \ref{cor-R0}, $M^m$ is a linear subspace.
\end{proof}

By using Theorem \ref{thm-RW} and Theorem \ref{thm-volume}, we can improve Corollary \ref{cor-R1} as follows

\begin{cor}\label{cor-R2}

Let $X: M^m \to \mathbb{R}^{m+n}_n$ be an $m$-dimensional spacelike entire self-shrinking graph. Assume that the origin $o\in M^m$ and the $w$-function satisfies
\begin{equation*}
\limsup_{x \to \infty} \frac{\log w}{z} < \frac{1}{4}.
\end{equation*}
Then $M^m$ has to be a linear subspace.
\end{cor}

\begin{proof} 
Suppose  that $M^m$ is not a linear subspace.
Let $f(R)=\max_{\{z=R\}}w$, since $\Delta_{V}\log w \geq \frac{\|B\|^2}{w^2}> 0$, then the maximum principle implies that $f(R)$ is nondecreasing in $R$. If $f$ is bounded by some positive constant, then by using Theorem \ref{thm-volume} and the assumption, we have
\beq\label{eqn-finite}
\int_{M^m} w^2 (\log w)^2 e^{-\frac{z}{4}} < \infty.
\eeq
Otherwise, $\lim\limits_{R\to \infty} f(R) = \infty$, then for any $\epsilon>0$, we obtain $\log f(R) \leq f(R)^{\frac{\epsilon}{2}}$ when $R$ is large. Therefore by the assumption, we can conclude that 
\begin{align*}
   f(R) < e^{\frac{1-\epsilon}{4}R}
\end{align*}
for $R$ large enough. It follows that 
\begin{align*}
  f(R)(\log f(R))^2\leq f(R)^{1+\epsilon} \leq e^{\frac{1-\epsilon^2}{4}R}, \quad {\rm for} \quad R \quad {\rm large}.  
\end{align*}
Then by Theorem \ref{thm-volume} again, we can also obtain \eqref{eqn-finite}. Since $M^m$ is not a linear subspace, by \eqref{eqn-H8}, $w$ can not be a constant, in particular, $w \not\equiv 1$, that is, $\int_{M^m} w^2 (\log w)^2 e^{-\frac{z}{4}} \neq 0$. Therefore, we have
\[
\limsup_{R\to \infty} \frac{(\log R)^2}{\int_{D_R} w^2 (\log w)^2 e^{-\frac{z}{4}}} = \infty.
\]
This is a contradiction with \eqref{eqn-H10}. 
\end{proof}

By \eqref{eqn-w-function}, Corollary \ref{cor-R2} can be rewritten as

\begin{cor}\label{cor-R3}

Let $M^m:= \{ \left. (x, u(x)) \right| x \in \mathbb{R}^m, u=(u^1, u^2, ..., u^n)  \}$ be an $m$-dimensional  spacelike entire self-shrinking graph in $\mathbb{R}^{m+n}_n$. Assume the origin $o \in M^m$ and the induced metric $(g_{ij})$ satisfies
\begin{align*}
\liminf_{|x|\to \infty} \dfrac{\log \det(g_{ij}(x))}{|x|^2-|u(x)|^2} > -\dfrac{1}{2},
\end{align*}
where $g_{ij}(x) =\d_{ij} -\sum_{\alpha=1}^n u^\a_{i}(x)u^{\a}_j(x)$.
Then $M^m$ is a linear subspace.
\end{cor}

We are now in position to show that Corollary \ref{cor-R3} improves Theorem 3 in \cite{DW} as follows

\begin{cor}[\cite{DW}]\label{cor-DW}
Let $M^m:= \{ \left. (x, u(x)) \right| x \in \mathbb{R}^m, u=(u^1, u^2, ..., u^n)  \}$ be an $m$-dimensional  spacelike entire self-shrinking graph in $\mathbb{R}^{m+n}_n$. Assume the origin $o \in M^m$ and the induced metric $(g_{ij})$ satisfies
\begin{align}\label{eq:DW}
\lim_{|x|\to \infty} \dfrac{\log {\rm det} (g_{ij}(x))}{|x|}=0,
\end{align}
where $g_{ij}(x) =\d_{ij} -\sum_{\alpha=1}^n u^\a_{i}(x)u^{\a}_j(x)$. Then $M$ is a linear subspace.
\end{cor}
\begin{proof}Since $\det(g_{ij})<1$, Ding-Wang's assumption \eqref{eq:DW} implies that for every positive constant $\varepsilon$, we have
\begin{align*}
    \dfrac{\log \det (g_{ij})}{|x|}>-\varepsilon,\quad \text{as}\ |x|\to\infty.
\end{align*}
By (\ref{eqn-pseudo}), the function $z$ is proper, choosing $\varepsilon=\frac{1}{4C_1}$, then we obtain
\begin{align*}
    \dfrac{\log \det (g_{ij})}{z}\geq 2C_1\dfrac{\log \det (g_{ij})}{|x|}>-2C_1\varepsilon=-\dfrac{1}{2},\quad \text{as}\ |x|\to\infty.
\end{align*}
Then this Corollary follows from Corollary \ref{cor-R3}.
\end{proof}

\begin{rem}
If $m=1$, then the growth condition is not necessary. In other words, the only entire graphic spacelike self-shrinking curve through the origin in the pseudo-Euclidean space $\mathbb{R}^{1+n}_n$ has to be a linear subspace. In fact, assume that $M^1=\{(t,u^1(t),\dotsc,u^n(t)): t\in\mathbb{R}\}$ is a spacelike self-shrinking curve, then
\begin{align}\label{eq:spacelikecurve}
    \dfrac{u^{\alpha}_{tt}}{1-\sum_{\beta=1}^nu^{\beta}_tu^{\beta}_t}=\dfrac{1}{2}\left(tu_{t}^{\alpha}-u^{\alpha}\right),\quad\forall t\in\mathbb{R},\ \quad \alpha=1,\dotsc,n.
\end{align}
Since $M^1$ contains the origin, we have $u^{1}(0)=\dotsm=u^{n}(0)=0$. Denote by $u_{t}^{\alpha}(0)=a^{\alpha}, \alpha=1,\dotsc,n$, then $\{u^{\alpha}(t)=a^{\alpha}t, \alpha=1,\dotsc,n\}$ is a solution to \eqref{eq:spacelikecurve} and $M^1$ is a linear subsapce.   By the uniqueness theorem of ODE system, we know that $M^1$ has to be a linear subspace.   
\end{rem}

At the end of this section, we shall give a nontrivial  spacelike entire self-shrinking graph which does not contain the origin (cf. \cite{HOW}). 

\begin{eg} Consider a $C^2$ function $u: \mathbb{R} \to \mathbb{R}$ satisfying 
\beq\label{eqn-ODE}
\dfrac{u''}{1-u'^{2}} = \dfrac{1}{2}(tu'-u), \quad |u'|<1.
\eeq
If we find a nontrivial solution $u$ to \eqref{eqn-ODE}, i.e., $u$ is not a linear function, then
\begin{align*}
    M^m=\left\{\left(x_1,x_2,\dotsc,x_m,u(x_1),0,\dotsc,0\right)\in\mathbb{R}_n^{m+n}: \left(x_1,x_2,\dotsc,x_m\right)\in\mathbb{R}^m\right\}
\end{align*}
is a nontrivial entire spacelike self-shrinking graph in $\mathbb{R}^{m+n}_n$, i.e., this graph is not an affine plane. According to Chen-Qiu's result (\cite{CQ}), this entire graphic self-shrinker can not be complete.

Indeed, we consider the following ODE
\begin{equation*} 
\left\{
         \begin{array}{rcl}
           \ds\vs & w''=  \frac{1}{2}(tw'-w)(1-w'^2),\\
     \ds & w(0)= a\in(-\infty,0), \quad w'(0)=b \in (-1, 1).
         \end{array}
           \right.
\end{equation*}
Assume the maximal existence interval is $(-T_1, T_2)$ with $T_1, T_2 \in (0, \infty]$ such that $|w'(t)|< 1, \forall T\in (-T_1, T_2)$. Set $\phi=tw'-w$, then
\begin{align*}
    \phi'=\dfrac{t}{2}\left(1-w'^2\right)\phi.
\end{align*}
Consider a function $f'=\frac{t}{2}\left(1-w'^2\right), f(0)=0$, then $f\geq0$ and $e^{-f}\phi$ is a contant. In particular $\phi\geq-a$. Thus
\begin{align*}
    \left(\tanh^{-1}w'\right)'=\dfrac{1}{2}\left(tw'-w\right)\geq-\dfrac{a}{2}.
\end{align*}
For $t\in[0,T_2)$, we have
\begin{align}\label{eqn-ODE+1}
    w'\geq\tanh\left(-\dfrac{at}{2}+\tanh^{-1}b\right)=-\tanh\left(\dfrac{at}{2}-\tanh^{-1}b\right),
\end{align}
which implies
\begin{align*}
    w\geq&-\dfrac{2}{a}\log\cosh\left(-\dfrac{at}{2}+\tanh^{-1}b\right)+\dfrac{2}{a}\log\cosh\left(\tanh^{-1}b\right)+a\\
    \geq&-\dfrac{1}{a}\left(\log\cosh\left(-at\right)+\log\cosh\left(2\tanh^{-1}b\right)\right)+\dfrac{2\tanh^{-1}|b|}{a}+a\\
    \geq& t+\dfrac{2}{a}\log 2+a.
\end{align*}
Hence $\phi\leq -a-\frac{2}{a}\log 2$. Thus
\begin{align*}
    \left(\tanh^{-1}w'\right)'\leq-\dfrac{a}{2}-\dfrac{1}{a}\log 2,
\end{align*}
which implies
\begin{align}\label{eqn-ODE+2}
    w'\leq\tanh\left(\left(-\dfrac{a}{2}-\dfrac{1}{a}\log 2\right)t+\tanh^{-1}b\right).
\end{align}
Then \eqref{eqn-ODE+1} and \eqref{eqn-ODE+2} implies that $T_2=+\infty$.

For $t\in(-T_1,0)$, a similar argument gives
\begin{align*}
    -\tanh\left(\left(\dfrac{a}{2}+\dfrac{1}{a}\log 2\right)t-\tanh^{-1}b\right)\leq w'\leq\tanh\left(-\dfrac{at}{2}+\tanh^{-1}b\right),
\end{align*}
which implies that $T_1=-\infty$.

\end{eg}

\vskip12pt

The above example implies that for $|t|>-\frac{a}{2}$,
\begin{align*}
    \dfrac{\log\left(1-w'^2\right)}{t^2-w^2}\geq\dfrac{\log\left(1-w'^2\right)}{-2a|t|-a^2}\geq\dfrac{2\log\cosh\left(-\left(\frac{a}{2}+\frac{1}{a}\log 2\right)|t|+\tanh^{-1}|b|\right)}{2a|t|+a^2},
\end{align*}
and we get
\begin{align*}
    \liminf_{|t|\to\infty}\dfrac{\log\left(1-w'^2\right)}{t^2-w^2}\geq-\dfrac{1}{2}-\dfrac{\log 2}{a^2}\to-\dfrac{1}{2},\quad\text{as}\ a\to-\infty.
\end{align*}
Motivated by Corollary \ref{cor-R3} and the above example, we would like to propose the following
\begin{conj}
Let $u=\left(u^1, u^2, ..., u^n\right)$ be an entire smooth solution to
\begin{align*}
    \sum_{i,j=1}^mg^{ij}(x)u^{\alpha}_{ij}(x)=\dfrac12\left(\sum_{i=1}^{m}x_iu_i^{\alpha}(x)-u^{\alpha}(x)\right),\quad x\in\mathbb{R}^m,\quad \alpha=1,\dotsc,n,
\end{align*}
where $g_{ij}(x) =\d_{ij} -\sum_{\alpha=1}^n u^\a_{i}(x)u^{\a}_j(x)$ and $\left(g^{ij}(x)\right)_{1\leq i, j\leq m}$ is the inverse matrix of $\left(g_{ij}(x)\right)_{1\leq i, j\leq m}$. Assume $u^{1}(0)=\dotsm= u^m(0)=0$ and
\begin{align*}
\liminf_{|x|\to \infty} \dfrac{\log \det(g_{ij}(x))}{|x|^2-|u(x)|^2}\geq -\dfrac{1}{2},
\end{align*}
then $u^{\alpha}(x)$ are linear functions for each $\alpha=1,\dotsc,n$.
\end{conj}

\vskip24pt


\end{document}